\documentclass{article}
\usepackage{amsmath,amsthm,amssymb}
\newcommand{\nc}{\newcommand}

\newtheorem{thm}{Theorem}[section]
\newtheorem{rmk}[thm]{Remark}

\newtheorem{lemma}[thm]{Lemma}

\newtheorem{corollary}[thm]{Corollary}

\newtheorem{definition}[thm]{Definition}

\newenvironment{cor}{\begin{corollary} \rm}{\end{corollary}}

\nc{\Ext}{\operatorname{Ext}}
\nc{\FS}{\operatorname{FS}}
\nc{\NS}{\operatorname{NS}}
\nc{\Amp}{\operatorname{Amp}}
\nc{\Pic}{\operatorname{Pic}}
\nc{\Kom}{\operatorname{Kom}}
\nc{\DGrB}{\operatorname{DGrB}}
\nc{\antidiag}{\operatorname{antidiag}}
\nc{\diag}{\operatorname{diag}}
\nc{\mo}{\operatorname{mod}}
\nc{\Gr}{\operatorname{Gr}}
\nc{\Rep}{\operatorname{Rep}}
\nc{\Perf}{\operatorname{Perf}}
\nc{\Hom}{\operatorname{Hom}}
\nc{\Sym}{\operatorname{Sym}}
\nc{\RHom}{R\operatorname{Hom}}
\nc{\cRHom}{\operatorname{\mathcal{R}\mathcal{H}om}}
\nc{\cHom}{\operatorname{\mathcal{H}om}}
\nc{\End}{\operatorname{End}}
\nc{\Coh}{\operatorname{Coh}}
\nc{\Aut}{\operatorname{Aut}}
\nc{\Td}{\operatorname{Td}}
\nc{\Coker}{\operatornamoe{Coker}}
\nc{\coker}{\operatorname{coker}}
\nc{\colim}{\operatorname{colim}}
\nc{\Ker}{\operatorname{Ker}}
\nc{\img}{\operatorname{Im}}
\nc{\D}{\operatorname{D}}
\nc{\ch}{\operatorname{ch}}
\nc{\Stab}{\operatorname{Stab}}
\nc{\SL}{\operatorname{SL}}
\nc{\rk}{\operatorname{rk}}
\nc{\GL}{\operatorname{GL}}
\nc{\Log}{\mathop{\mathrm{Log}}}
\nc{\abs}[1]{\lvert#1\rvert}
\nc{\Cone}{\operatorname{Cone}}
\nc{\id}{\operatorname{id}}
\nc{\Li}{\operatorname{Li}}
\nc{\fmod}{\operatorname{fmod}}
\nc{\SdR}{\operatorname{SdR}}

\nc{\cA}{{\mathcal A}}
\nc{\cB}{{\mathcal B}}
\nc{\cC}{{\mathcal C}}
\nc{\cD}{{\mathcal D}}
\nc{\cE}{{\mathcal E}}
\nc{\cF}{{\mathcal F}}
\nc{\cG}{{\mathcal G}}
\nc{\cH}{{\mathcal H}}
\nc{\cI}{{\mathcal I}}
\nc{\cJ}{{\mathcal J}}
\nc{\cK}{{\mathcal K}}
\nc{\cL}{{\mathcal L}}
\nc{\cM}{{\mathcal M}}
\nc{\cN}{{\mathcal N}}
\nc{\cO}{{\mathcal O}}
\nc{\cP}{{\mathcal P}}
\nc{\cQ}{{\mathcal Q}}
\nc{\cR}{{\mathcal R}}
\nc{\cS}{{\mathcal S}}
\nc{\cT}{{\mathcal T}}
\nc{\cU}{{\mathcal U}}
\nc{\cV}{{\mathcal V}}
\nc{\cW}{{\mathcal W}}
\nc{\cX}{{\mathcal X}}
\nc{\cY}{{\mathcal Y}}
\nc{\cZ}{{\mathcal Z}}
 
\nc{\bA}{{\mathbb A}}
\nc{\bB}{{\mathbb B}}
\nc{\bC}{{\mathbb C}}
\nc{\bD}{{\mathbb D}}
\nc{\bE}{{\mathbb E}}
\nc{\bF}{{\mathbb F}}
\nc{\bG}{{\mathbb G}}
\nc{\bH}{{\mathbb H}}
\nc{\bI}{{\mathbb I}}
\nc{\bJ}{{\mathbb J}}
\nc{\bK}{{\mathbb K}}
\nc{\bL}{{\mathbb L}}
\nc{\bM}{{\mathbb M}}
\nc{\bN}{{\mathbb N}}
\nc{\bO}{{\mathbb O}}
\nc{\bP}{{\mathbb P}}
\nc{\bQ}{{\mathbb Q}}
\nc{\bR}{{\mathbb R}}
\nc{\bS}{{\mathbb S}}
\nc{\bT}{{\mathbb T}}
\nc{\bU}{{\mathbb U}}
\nc{\bV}{{\mathbb V}}
\nc{\bW}{{\mathbb W}}
\nc{\bX}{{\mathbb X}}
\nc{\bY}{{\mathbb Y}}
\nc{\bZ}{{\mathbb Z}}
    
\begin{document}
\title{On Euler characteristics for large Kronecker quivers}
\author{So Okada\footnote{Supported by JSPS Grant-in-Aid and Global
    Center of Excellence Program at Kyoto University, Email:
    okada@kurims.kyoto-u.ac.jp, Address: Masukawa Memorial Hall 502,
    Research Institute for Mathematical Sciences, Kyoto University,
    606-8502, Japan.}}
\maketitle
 
\begin{abstract}
  We study Euler characteristics of moduli spaces of stable
  representations of $m$-Kronecker quivers for $m>>0$.
\end{abstract}

\section{Introduction}
For each positive integer $m$, let $K^{m}$ be the $m$-Kronecker quiver
which consists of two vertices and $m$ arrows from one to the
other. For generic non-trivial stability conditions \cite{Bri} on the
category of representations of $K^{m}$ and moduli spaces of stable
representations $M(K^{m}(a,b))$ of coprime dimension vectors $(a,b)$
\cite{Kin}, we study Euler characteristics $\chi(K^{m}(a,b))$.

We put some more details in the later section and we go on as follows.
Notice that for the Euler form $\langle \cdot, \cdot \rangle$ and a
symplectic form $\{ \cdot, \cdot \}$, which is an anti-symmetrization
of the Euler form, we may take a non-trivial stability condition on
the category of representations of $K^{m}$ such that for
representations $E,F$ of $K^{m}$ and the slope function $\mu$, we have
$\mu(E)>\mu(F)$ if and only if $\{E,F\}>0$.
 
For objects to study in terms of wall-crossings, stability conditions
such that the positivity of the difference of slopes coincides with
that of symplectic forms on the Grothendieck group have been commonly
called Denef's stability conditions in physics \cite{Den00}.  We
employ these special stability conditions and the terminology.

Euler characteristics $\chi(K^{m}(a,b))$ have been studied
extensively.  In particular, formulas of Kontsevich-Soibelman and
Reineke \cite{KonSoi,Rei03,Rei10} have been known.  In this article,
we would like to study quantitative questions for $m>>0$.

To analyze further, for each coprime $a,b$ and $m>0$, let us define
the bipartite quiver $Q^{m}(a,b)$ which consist of $a$ source vertices
and $b$ terminal vertices with $m$ arrows from each source vertex to
each terminal vertex.  On representations of $Q^{m}(a,b)$, we have
Denef's stability conditions (see Section \ref{sec:proofs}).

We denote $M(Q^{m}(a,b))$ to be the moduli space of stable
representations of dimension vectors being one on each vertex of
$Q^{m}(a,b)$ and $\chi(Q^{m}(a,b))$ to be the corresponding Euler
characteristic.

In this paper, we first prove the following.
\begin{thm}\label{thm:eq} 
  For each coprime $a,b$, and $m>>0$, we have
  $$\chi(Q^{1}(a,b))\sim 
  \frac{a!b!}{m^{a+b-1}}
  \chi(K^{m}(a,b)).$$
\end{thm}
In terms of physics, we would like to mention that in Theorem
\ref{thm:eq}, Euler characteristics in the left-hand and right-hand
sides are discussed to be blackhole counting in supergravity
\cite{ManPioSen} and Witten index in superstring theory \cite{Den02}.
In \cite{Oka10}, with the framework of Kontsevich's homological mirror
symmetry \cite{Kon}, the $m$-Kronecker quiver $K^{m}$ has been
described in terms of Lagrangian intersection theory.

Key tools to obtain Theorem \ref{thm:eq} are the recently obtained
formula Theorem \ref{thm:mps} on $\chi(K^{m}(a,b))$ by
Manschot-Pioline-Sen \cite{ManPioSen} (MPS formula for short) and our
Lemma \ref{lem:linear}.  We realize that by taking $m$ to be a
variable, MPS formula provides the polynomial expansion of
$\chi(K^{m}(a,b))$ whose coefficients involve Euler characteristics of
bipartite quivers such as $Q^{m}(a,b)$. Indeed, we are dealing with
nothing but the first-order approximation of $\chi(K^{m}(a,b))$ for
$m>>0$.

By Theorem \ref{thm:eq}, to compute $\chi(Q^{1}(a,b))$, we can take
the advantage of $\chi(K^{m}(a,b))$.  Since the explicit formula of
$\chi(K^{m}(a,a+1))$ has been provided in \cite{Wei09}, we can obtain
$\chi(Q^{1}(a,a+1))$ by taking $m\to \infty$ in Corollary
\ref{cor:d+1}. Let us mention that for the cases of $a=1$ and
arbitrary $b$, we see that Stirling formula explains Theorem
\ref{thm:eq}.

When $a+b=1$, $M(K^{m}(a,b))$ is a point.  Taking logarithms in
Theorem \ref{thm:eq}, we have the following.
\begin{cor}\label{cor:douglas} 
  For $m>>0$, we have
  $$\ln(\chi(K^{m}(a,b)))\sim (a+b-1)\ln (m).$$
  In particular, for $a,b>>0$ such that $\frac{b}{a} \sim r$ and
  large enough $m$ depending on $a,b$, we have
  $$\frac{\ln(\chi(K^{m}(a,b)))}{a}\sim (1+r)\ln (m).$$
\end{cor}

Let us mention that Douglas has conjectured the following
\cite{GroPan,Wei09}.  For coprime $a,b>>0$ such that $\frac{b}{a}\sim
r$ and each $m$, we have that $\frac{\ln \chi(K^{m}(a,b))}{a}$ is a
continuous function of $r$.  In \cite{Wei09}, assuming the continuity,
the quantity has been determined through the explicit formula of
$\chi(K^{m}(a,a+1))$ in \cite{Wei09} mentioned above.  So, now we
notice that some estimates on
$\frac{\ln(\frac{\chi(Q^{1}(a,b))}{a!b!})}{a+b-1}$ in terms of
$\frac{b}{a}\sim r$ for $a>>0$ would give further understanding of
$\frac{\ln\chi(K^{m}(a,b))}{a}$.

\section{Proofs}\label{sec:proofs}
Let us expand and introduce notions.  For each $a$, let $\overline{a}$
denote a partition of $a$ such that for non-negative integers $a_{l}$
of $l\geq 1$, we have $\sum_{l} l a_{l}=a$. We put
$S_{\overline{a}}=\sum a_{l}$ for our convenience.  When $a_{1}=a$, we
simply write $a$ for $\overline{a}$.  For a quiver $Q$ and
representations $E,F$ of $Q$, on the Grothendieck group of the
category of representations of $Q$, let $\langle E, F \rangle_{Q}$ be
the Euler form and $\{ E,F \}_{Q}$ be the symplectic form $\langle
F,E\rangle_{Q}-\langle E,F\rangle_{Q}$.  For a dimension vector $d$,
we call a partition $d^{1},\ldots,d^{s}$ of $d$ such that
$\sum_{p=1}^{s}d^{p}=d$ and $\{\sum_{p=1}^{b}d^{p},d \}_{Q}>0$ for
each $b=1,\ldots, s-1$ to be admissible.

For each $m>0$ and partitions $\overline{a},\overline{b}$ of $a$ and
$b$, we define the bipartite quiver $Q^{m}(\overline{a},\overline{b})$
as follows.  It consists of $S_{\overline{a}}$ source vertices such
that for each $l$, we have $a_{l}$ vertices $v$; for our convenience,
we say $a_{l}$ is the label of $v$ and we put $w(v)=l$.  It consists
of $S_{\overline{b}}$ terminal vertices with labels and $w(\cdot)$
being defined by the same manner.  We put $m w(v) w(v')$ arrows from
each source vertex $v$ to each terminal vertex $v'$.
  
Let us explain Denef's stability conditions in use.  For the
$m$-Kronecker quiver $K^{m}$, the source vertex $(1,0)$, and the
terminal vertex $(0,1)$, the slope function $\mu$ satisfies
$\mu(1,0)>\mu(0,1)$.  For $Q^{m}(\overline{a},\overline{b})$ and
vertices $v$ and $v'$ with the labels being $a_{l}$ and $b_{l'}$,
central charges $\frac{Z(v)}{w(v)}$ and $\frac{Z(v')}{w(v')}$ coincide
with those of the vertices $(1,0)$ and $(0,1)$.

We write $(\overline{a},\overline{b})$ for the dimension vector which
has one on each vertex of the quiver
$Q^{m}(\overline{a},\overline{b})$. We let
$M(Q^{m}(\overline{a},\overline{b}))$ to be the moduli space of stable
representations of the dimension vector $(\overline{a},\overline{b})$
of $Q^{m}(\overline{a},\overline{b})$.  For coprime $a,b$ and moduli
spaces of stable representations of quivers such that $M(K^{m}(a,b))$
and $M(Q^{m}(\overline{a},\overline{b}))$, we denote $P(K^{m}(a,b),y)$
and $P(Q^{m}(\overline{a},\overline{b}),y)$ to be Poincare
polynomials.  For the $m$-Kronecker quiver $K^{m}$, we have the
following MPS formula \cite[Appendix D]{ManPioSen}.

\begin{thm}\label{thm:mps}(MPS formula)
  For each coprime $a,b$ and $m>0$, we have
\begin{align*}
  P(K^{m}(a,b),y)&=
  y^{-\langle (a,b),(a,b) \rangle_{K^{m}}}
  \sum_{\overline{a},\overline{b}} y^{ \langle
    (\overline{a},\overline{b}), 
    (\overline{a},\overline{b}) 
    \rangle_{Q^{m}(\overline{a},\overline{b})}}
  P(Q^{m}(\overline{a},\overline{b}),y)\cdot \\ 
  & \ \ \  \ \  \Pi_{l} \frac{1}{\overline{a}_{l}!}
  \left(\frac{y-y^{-1}}{l(y^{l}-y^{-l})}(-1)^{l-1}\right)^{\overline{a}_{l}} \cdot \\
  &  \ \ \ \ \ \Pi_{l} \frac{1}{\overline{b}_{l}!}
  \left(\frac{y-y^{-1}}{l(y^{l}-y^{-l})}(-1)^{l-1}\right)^{\overline{b}_{l}}.
\end{align*}
\end{thm}

We shall not repeat their proof of MPS formula, but we would like to
mention a key point of the proof as follows.  To compute
$P(Q^{m}(\overline{a},\overline{b}),y)$ with Reineke's formula
\cite[Corollary 6.8]{Rei03}, we start with a partition
$(\alpha^{p},\beta^{p})$ of $(\overline{a},\overline{b})$ for
$p=1,\ldots,s$ of some $s$.  For each $p$ and $l$, we put
$\alpha_{l}^{p}$ to denote the number of non-zero entries of
$\alpha^{p}$ and of vertices of labels being $a_{l}$; we put
$\beta_{l}^{p}$ of $\beta^{p}$ by the same manner.  Observe that,
through direct computation on symplectic forms, the partition
$(\alpha^{p},\beta^{p})$ is admissible if and only if the partition
$(\sum_{l}l \alpha_{l}^{p}, \sum_{l} l \beta_{l}^{p})$ of $(a,b)$ for
$p=1,\ldots,s$ is admissible.  Now, as shown in \cite[Appendix
D]{ManPioSen}, we can proceed by explicitly computing involved terms
for admissible partitions.

For Euler characteristics, we put the following for our convenience.
\begin{cor}\label{cor:mps}
We have
$$\chi(K^{m}(a,b))=
\sum_{\overline{a},\overline{b}} \chi(Q^{m}(\overline{a},\overline{b}))
\cdot \Pi_{l} \frac{1}{\overline{a}_{l}!} \frac{(-1)^{\overline{a}_{l}(l-1)}}{l^{2 \overline{a}_{l}}}
\cdot \Pi_{l} \frac{1}{\overline{b}_{l}!} \frac{(-1)^{\overline{b}_{l}(l-1)}}{l^{2 \overline{b}_{l}}}.$$
\end{cor}

Notice that $M(Q^{1}(\overline{a},\overline{b}))$ is a non-trivial
smooth projective variety, since we have stable representations
including ones with invertible maps on every arrows.  Now, we have 
the following.

\begin{lemma}\label{lem:linear}
  We have
  $$\chi(Q^{m}(\overline{a},\overline{b}))=m^{S_{\overline{a}}+S_{\overline{b}}-1}
  \chi(Q^{1}(\overline{a},\overline{b})).$$
\end{lemma}
\begin{proof}
  Let us consider the Poincare polynomial
  $P(Q^{m}(\overline{a},\overline{b}),y)$ with Reineke's formula
  \cite[Corollary 6.8]{Rei03}.  For the dimension vector
  $(\overline{a},\overline{b})$, we take an admissible partition
  $d^{1},\ldots, d^{s}$ and the term $(-1)^{s-1}y^{2\sum_{k\leq l}
    \sum_{v\to v'}d_{v}^{l}d_{v'}^{k}}$.

  We notice that $\{ \cdot, \cdot
  \}_{Q^{m}(\overline{a},\overline{b})}=m \{ \cdot, \cdot
  \}_{Q^{1}(\overline{a},\overline{b})}$.  The set of admissible
  partitions is invariant under choices of $m$. For each admissible
  partition, the power of $y$ above is the $m$ times of that for
  $P(Q^{1}(\overline{a},\overline{b}),y)$.

  We have that $P(Q^{1}(\overline{a},\overline{b}),y)$ is a non-zero
  polynomial.  Ignoring an overall factor of a power of ¡¡$y$ and
  writing $y^{2}$ as $q$ for simplicity, for some non-trivial and
  non-negative integers $\alpha_{i}$ and $\beta_{i}$, we have
  \begin{align*}
  P(Q^{1}(\overline{a},\overline{b}),q)
  &=(q-1)^{1-S_{\overline{a}}-S_{\overline{b}}} \left( \sum_{i\geq 0} \alpha_i
  (q-1)^{S_{\overline{a}}+S_{\overline{b}}-1} q^{\beta_i}\right).
  \end{align*}
  For admissible partitions, the second factor is the sum of terms
  above.  So we have
  \begin{align*}
    P(Q^{m}(\overline{a},\overline{b}),q)
    &=(q-1)^{1-S_{\overline{a}}-S_{\overline{b}}} \left( \sum_{i \geq 0} \alpha_i
      (q^{m}-1)^{S_{\overline{a}}+S_{\overline{b}}-1} q^{m \beta_i}\right), 
 \end{align*}
 and the assertion follows.
\end{proof}

We put a proof of Theorem \ref{thm:eq}.
\begin{proof}
  By Lemma \ref{lem:linear}, we see that the term $\chi(Q^{m}(a,b))$
  carries the highest power of $m$ among
  $\chi(Q^{m}(\overline{a},\overline{b}))$ in Corollary \ref{cor:mps}.
\end{proof}
      
We put a proof of Corollary \ref{cor:douglas}.
\begin{proof}\label{proof:douglas}
  We have
  \begin{align*}
    \ln\left(\frac{\chi(Q^{m}(a,b))}{a!b!}\right)
    &=\ln\left(\frac{m^{a+b-1}\chi(Q^{1}(a,b))}{a!b!}\right)\\
    &=(a+b-1)\ln (m)
    +\ln\left(\frac{\chi(Q^{1}(a,b))}{a!b!}\right).
  \end{align*}
  So for $a+b\neq 1$ and large enough $m$ so that
  \begin{align*}
    \left|\frac{\ln(\frac{\chi(Q^{1}(a,b))}{a!b!})}{(a+b-1)\ln(m)}\right|&<<1,\\
  \end{align*}
  the assertion follows.
\end{proof}
      
Let us compute $\chi(Q^{1}(a,a+1))$ as in the introduction.  
From \cite{Wei09}, we recall the following.
\begin{thm}(\cite[Theorem 6.6]{Wei09})
  $$\chi(K^{m}(a,a+1))
  =\frac{m}{(a+1)((m-1)a+m)} \binom{(m-1)^{2}a+(m-1)m}{a}.$$
\end{thm}
        
So, by Theorem \ref{thm:eq}, we have the following.
\begin{cor}\label{cor:d+1}
  $$\chi(Q^{1}(a,a+1))
  =\lim_{m\to \infty}\frac{\chi(K^{m}(a,a+1)) a! (a+1)!}{m^{2a}}
  =(a+1)!(a+1)^{-2+a}.$$
\end{cor}
 
\begin{rmk}
  With the formula of $\chi(K^{m}(2,2a+1))$ in \cite{Rei03}, Manschot
  told the author that he has proved the following formula.
  $$\chi(Q^{1}(2,2a+1))=\frac{(2a+1)!}{a!^2}.$$
  This sequence and the one in Corollary \ref{cor:d+1} coincide with
  A002457 and A066319 at oeis.org.
\end{rmk}

\section*{Acknowledgments}
The author thanks Research Institute for Mathematical Sciences and
Institut des Hautes \'Etudes Scientifiques for providing him with
excellent research environment. He thanks Professors Dimofte, Feigin,
Fukaya, Galkin, Kajigaya, I. Kimura, Kontsevich, Manschot, Matsuki,
Nakajima, Pioline, Y. Sano, A. Sasaki, Stoppa, Tachikawa, Weinberger,
Weist, and Zagier for their comments, discussions, or lectures
related to the paper. In particular, he thanks Professor Weist 
for his discussion.  He thanks Professor Hosomichi at Yukawa 
Institute of Theoretical Physics for organizing a seminar by Professor 
Manschot in January 2011, when he first heard about MPS formula, and 
organizers of the summer school and the conference dedicated to the 
60th birthday of Professor Mori on minimal models and extremal rays 
for exciting talks related to the paper.

\end{document}